\documentclass[12pt]{article}
\usepackage{amsmath}
\usepackage{amsthm,amsfonts,amssymb}
\allowdisplaybreaks[1]

\usepackage{graphicx}
\usepackage{hyperref}
\usepackage{indentfirst}
\usepackage{mathrsfs}
\usepackage{amscd}
\usepackage{caption2}
\usepackage{amsfonts}
\usepackage{epsfig}
\usepackage{graphics}
\usepackage{dsfont}
\usepackage{bbm}
\usepackage{fullpage}
\usepackage[numbers]{natbib}
\usepackage[T1]{fontenc}
\usepackage[utf8]{inputenc}
\usepackage{authblk}
\usepackage{setspace}
\usepackage{algpseudocode}

\newtheorem{theorem}{Theorem}[section]
\newtheorem{lemma}[theorem]{Lemma}

\newtheorem{remark}[theorem]{Remark}

\title{Optimal routing in two-queue polling systems}

\author{
I.J.B.F. Adan\thanks{Department of Industrial Engineering, Eindhoven University of Technology, P.O. Box 513, 5600MB Eindhoven, The Netherlands, \href{mailto:iadan@tue.nl}{iadan@tue.nl}}
\quad
V.G. Kulkarni\thanks{Department of Statistics and Operations Research, University of North Carolina, Chapel Hill, NC 27599, USA, \href{mailto:vkulkarn@email.unc.edu}{vkulkarn@email.unc.edu},
Research supported by Visitor's Travel Grant of the Netherlands Organization for Scientific Research (NWO)} \quad
N. Lee\thanks{Department of Statistics and Operations Research, University of North Carolina, Chapel Hill, NC 27599, USA}
\quad
A.A.J. Lefeber
\thanks{Department of Mechanical Engineering, Eindhoven University of Technology, P.O. Box 513, 5600MB Eindhoven, The Netherlands, \href{mailto:a.a.j.lefeber@tue.nl}{a.a.j.lefeber@tue.nl}}
}

\date{}

\begin{document}

\maketitle

\begin{abstract}
We consider a polling system with two queues, exhaustive service, no switch-over times and exponential service times with rate $\mu$ in each queue. The waiting cost depends on the position of the queue relative to the server: It costs a customer $c$ per time unit to wait in the busy queue (where the server is) and $d$ per time unit in the idle queue (where no server is). Customers arrive according to a Poisson process with rate $\lambda$. We study the control problem of how arrivals should be routed to the two queues in order to minimize expected waiting costs and characterize individually and socially optimal routing policies under three scenarios of available information at decision epochs: no, partial and complete information. In the complete information case, we develop a new iterative algorithm to determine individually optimal policies, and show that such policies can be described by a switching curve. We conjecture that a linear switching curve is socially optimal, and prove that this policy is indeed optimal for the fluid version of the two-queue polling system.\\

\noindent
{\it Keywords:} Customer routing; dynamic programming; fluid queue; linear quadratic regulator; Nash equilibrium; polling system; Ricatti equation; social optimum.\\

\noindent{\it 2010 Mathematics subject classification:} Primary 60K25; Secondary 90B22.
\end{abstract}

\section{Introduction}
Polling systems have applications in diverse fields such as manufacturing, telecommunications, time-sharing computer systems, wireless networks, to name a few. There is a very large body of research work in polling systems, and we refer the readers to a few survey papers for the full range of issues in polling systems that researchers have studied: see the book by Takagi~\cite{T86}, and the review papers of Levy and Sidi~\cite{LS90}, Vishnevskii and Semenova ~\cite{VS06} and Boon et al.~\cite{BM11}.

The early work considered a simple polling system consisting of a single server serving $N$ queues in an exhaustive cyclic fashion, which means that it serves the customers in the $i$-th queue until it becomes empty and then moves to queue $i+1$ (or 1 if $i=N$). Results were obtained about the limiting distribution of the number of customers in the $N$ queues, their means, and waiting times, and so on. These results were quickly extended to service policies other than exhaustive, such as e.g. gated, $k$-limited and Bernoulli, as well as non-cyclic server routing, non-zero switch over times, and so on. We refer the reader to the sources mentioned above for the detailed references.

The issues of control of polling systems have received less attention than the performance analysis of polling systems. There are several possible control problems arising in polling systems. First, the order in which the queues are served can be determined to optimize system performance (such as weighted expected waiting times), assuming that the service discipline is fixed (such as exhaustive, or gated); see Boxma et al.~\cite{BL91}, Yechiali~\cite{Y93}, van der Wal and Yechiali~\cite{VY03}. When the server can switch after every service, the optimal dynamic service order is studied in greater detail, and may lead to simple rules like the $c\mu$ rule; for example, see Klimov~\cite{K74,K78}, Haijema and van der Wal~\cite{H08}. We refer to
Vishnevskii and Semenova~\cite{VS06} for many more papers in this area.

Customer routing in polling systems is a less studied area. Takine et al.~\cite{TT91}, Sidi et al.~\cite{SL92} and Boon et al.~\cite{BM13} study a Jackson network style routing of customers among $N$ queues, served cyclically by a single server. The control of customer routing is the focus of the current paper. This subject is also less studied in comparison with the control of server routing. To the best of our knowledge, the only paper on this topic is by Sharafali et al~\cite{SC04}. They consider the problem where the customers arriving at one of the queues can be routed to any of the others, while customers arriving at the other queues have no flexibility. They study static randomized routing policies and study the optimal fraction to be routed to each queue in order to minimize the weighted expected waiting cost. This comes somewhat close to our model that we will describe in the next section. However, our cost model is very different from the one in \cite{SC04} and we consider optimal policies under several scenarios of availability of information, and who is controlling the system. We consider a cyclic exhaustive polling system where every arriving customer needs to be routed to one of the queues, and distinguish three levels of information knowledge:

\begin{enumerate}
\item[i.] {\em No Information.} We do not know the queue lengths or the position of the server (that is, which queue the server is serving) at decision epochs.
\item[ii.] {\em Partial Information.} We know where the server is, but not the queue lengths, at decision epochs.
\item[iii.] {\em Complete Information.} We know the position of the server, and the queue lengths, at decision epochs.
\end{enumerate}

We assume that the waiting cost in a queue depends on its downstream position from the server. This is motivated by the tradeoff one normally encounters: we might be able to reduce the waiting cost by joining a queue farther from the server, but it will increase the total waiting time.  One motivation for such a case arises from appointment systems, although the analogy is not quite accurate:  one can be a walk-in customer and join the queue where the server is now (today's queue at the clinic), or one can get an appointment for a later day and join a shorter queue, because waiting at home is cheaper than waiting at the clinic.

Finally, we consider the control problem from two different points of view: the customer (or individually optimal) and the system manager (socially optimal) point of view .  Socially versus individually optimal policies have been well studied in the queueing literature, see Lippman and Stidham~\cite{LS77} and Hassin and Haviv~\cite{HH03}. Computation of individually optimal policies becomes complicated when the decisions of later customers can influence a customer's waiting cost. For example,  Altman and Shimkon~\cite{AS98} study individually optimal policies in  processor sharing queues, where the decisions by later customers influence one's waiting costs, because they affect the effective service rate available to any customer. They also introduce an iterative algorithm to find Nash equilibria. In our case, the analysis of individually optimal policies is similarly complicated by the fact that a customer's total cost is affected by the behavior of the  customers arriving after her. We provide a new iterative algorithm to derive Nash equilibria in such a case.

To keep the analysis simple, we consider an exponential system with only two queues and no switch-over times. Even for such a simple system, the analysis provides interesting insights, and can be quite involved. We introduce the model and notation in Section \ref{sec:mod}. The case of no information is studied in Section \ref{sec:noinfo}, and partial information in Section \ref{sec:partinfo}. In both cases, we study Nash equilibria in the individually optimal analysis, and minimize the long-run average cost rate in the socially optimal case. The case of complete information is studied in Section \ref{sec:cominfo}.  We present a new iterative algorithm to determine individually optimal policies, and show that such policies can be described by a switching curve. The socially optimal policies can be derived by using negative dynamic programming. We present a novel proof of existence of average cost optimal policies, but we have not been able to derive structural results in that case. However, numerical experimentation suggests a simple approximate socially optimal routing policy, which can described by a linear switching curve. In Section \ref{sec:fluid} we formulate the problem as a control problem of a fluid polling queueing system, and prove that the approximate policy conjectured in Section \ref{sec:cominfo} is in fact optimal in the fluid model. We end the paper with a numerical example and a brief summary in Section \ref{sec:num}.

\section{Polling Model}
\label{sec:mod}

We consider a polling system with two queues; see Figure~\ref{fig:model}.
\begin{figure}[ht]
\centering
\includegraphics[width=0.6\linewidth]{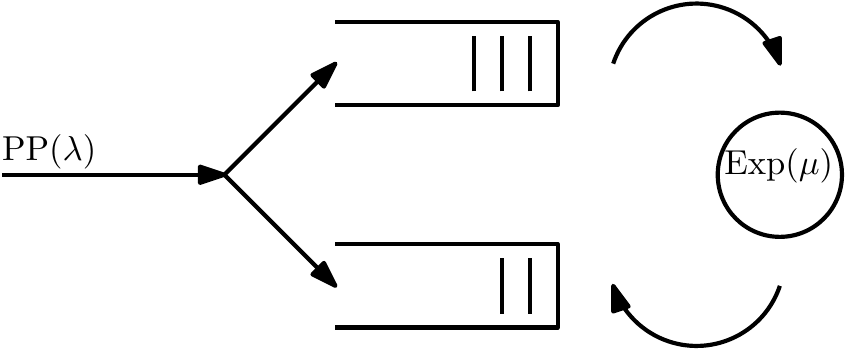}
\caption{The two-queue exponential polling model with exhaustive service.}\label{fig:model}
\end{figure}
Customers arrive
at this system according to a Poisson process with rate $\lambda$. The service times are
independent and exponentially distributed with rate $\mu$ at each queue.  A single server serves the two nodes in a cyclic
fashion with exhaustive service.  The switch-over times are assumed
to be zero. For stability we assume $\rho = \lambda/\mu < 1$. The only costs in the system are waiting costs: It costs a customer $c$ dollars to wait in a queue that is being served (called the busy queue), while it costs her
$d$ per unit time to wait in a queue that is not being served
(called the idle queue). We aim to study how the arrivals should be
routed to the two queues in order to minimize expected waiting costs. In the following sections we characterize individually optimal and socially optimal routing policies under various levels of information knowledge: No information, partial information and complete information.

\section{No Information}
\label{sec:noinfo}

Suppose that arriving customers have no information about the state of the
system, i.e., they do not know where the server is and what the
queue lengths are. In this case, the most general
policy  is described by a single parameter $p \in [0,1]$ as follows:
each customer joins queue $1$ with probability $p_1 = p$ and queue $2$
with probability $p_2 = 1-p_1$.  We use the notation $\rho_i = \rho p_i$ ($i=1,2$). Define $L_{ij}$ as the
expected number of customers  in queue $i$ given that the server is
serving queue $j$.  The next theorem gives these quantities; see e.g. Winands et al.~\cite{W06}, Winands~\cite{W07} and Boon~\cite{B11}.

\begin{theorem}\label{th:LB}
Under the above routing policy we have, for all $0 \le p \le 1$,
\begin{eqnarray}
L_{11} & = & \frac{\rho_1(1-\rho +\rho_1\rho_2 + \rho_2^2)}{(1-\rho)(1 - \rho + 2\rho_1\rho_2)} + 1, \\
L_{12} & = & \frac{\rho_1(\rho_1\rho_2 +(1-\rho_1)^2)}{(1-\rho)(1 - \rho + 2\rho_1\rho_2)}, \\
L_{21} & = & \frac{\rho_2(\rho_1\rho_2 +(1-\rho_2)^2)}{(1-\rho)(1 - \rho + 2\rho_1\rho_2)}, \\
L_{22} & = & \frac{\rho_2(1-\rho +\rho_1\rho_2 + \rho_1^2)}{(1-\rho)(1 - \rho + 2\rho_1\rho_2)} + 1.
\end{eqnarray}
\end{theorem}

Using the above theorem we derive the socially optimal policies in the next theorem.

\begin{theorem}\label{th:nosoc} {\bf (Socially Optimal Policies)}
\begin{enumerate}
\item[i.]
If $c > d$, there is a unique socially optimal policy: $p = \frac12$ (join either queue with probability $\frac12$).
\item[ii.]
If $c = d$, all policies $p \in [0,1]$ are socially optimal.
\item[iii.]
If $c < d$, there are two socially optimal policies: $p = 0$ and $p = 1$ (everyone joins queue 1 or everyone joins queue 2).
\end{enumerate}
\end{theorem}

\begin{proof}
The socially optimal policy minimizes the expected cost of a customer in steady state,
which is given by
\[ C  = p_1 C_1 + p_2 C_2 ,
\]
where $C_1$ and $C_2$ are the expected cost of joining queue $1$ and $2$, respectively,
\begin{eqnarray*}
C_1 &=& \frac{1}{\mu}\left(c(\rho_1L_{11}+\rho_2L_{12}) + d\frac{\rho_2}{1-\rho_2}L_{22}+c\right), \\
C_2 &=& \frac{1}{\mu}\left(c(\rho_1L_{21} + \rho_2L_{22}) + d\frac{\rho_1}{1-\rho_1}L_{11}+c\right) .
\end{eqnarray*}
It then follows that
\[ \frac{dC}{dp_1} = \frac{c-d}{\mu}\cdot \frac{\rho(\rho^2 + 2(1-\rho))(2p_1-1)}{(1-\rho +2\rho_1\rho_2)^2}.\]
Now if $c > d$, the above expression implies that $C$ decreases as
$p_1$ increases from 0 to $\frac12$ and increases as $p_1$ increases from
$\frac12$ to 1. Hence $C$ is minimized at $p_1 = \frac12$. If $c < d$, $C$
increases as $p_1$ increases from 0 to $\frac12$, and decreases as $p_1$
increases from $\frac12$ to 1. Thus $C$ is minimized at $p_1 =0$ and $p_1 =
1$. From symmetry, both these minima are identical. When $c=d$, $C$
does not depend on $p_1$. Hence the
result follows.
\end{proof}

The next theorem states the results about the individually optimal
(Nash equilibrium) policies.

\begin{theorem}\label{th:noNash} {\bf (Nash Equilibrium policies)}
\begin{enumerate}
\item[i.]
If $c(1-\rho) > d$, there is a unique Nash equilibrium policy: $p = \frac12$. This policy is socially optimal.
\item[ii.]
If $c(1-\rho) = d$, every policy $p \in [0,1]$ is a Nash equilibrium
policy, but only $p = \frac12$ is socially
optimal.
\item[iii.]
If $c(1-\rho) < d$, there are three Nash equilibrium policies: $p = 0$, $p=\frac12$ and
$p = 1$. Only policy $p = \frac12$ is socially
optimal if $c > d$, each of them is socially optimal if $c = d$, and $p = 0$ and
$p = 1$ are socially optimal if $c < d$.
\end{enumerate}
\end{theorem}

\begin{proof}
Suppose arriving customers join queue 1 with
probability $p_1 = p$ and  queue 2 with probability $p_2=1-p_1$. Now
suppose a smart customer knows how the other customers are behaving,
and decides to use this system to minimize her own waiting costs. If
she joins  queue 1, her expected cost is $C_1$ and otherwise,
if she joins  queue 2, her expected cost is $C_2$.
It then follows that
\[ C_1 - C_2 = \frac{\rho(1-2p_1)(d-c(1-\rho))}{\mu(1-\rho)(1-\rho +2\rho_1\rho_2)}.\]
Consider the case $c (1-\rho)> d$.
If all customers use $p_1 > \frac12$, then $C_1 > C_2$ and the smart customer will join
queue 2, that is, use $p_1 = 0$; and if all customers use $p_1 <
\frac12$, she will use $p_1 = 1$. Thus in these cases there is no Nash
equilibrium. If all the customers follow the policy
$p_1 = \frac12$, the smart customer is indifferent between the two
options and can choose $p_1 = \frac12$. Thus $p_1 = \frac12$ is a Nash
equilibrium. If $c (1-\rho) = d$, the smart customer is also indifferent, so all policies are a Nash equilibrium.
Next, in case $c(1-\rho) < d$, it is readily verified that there are three Nash equilibrium policies, $p_1=0$, $p_1 = \frac12$ and $p_1 =1$.
Together with Theorem \ref{th:nosoc}, this concludes the proof.
\end{proof}

\section{Partial Information}
\label{sec:partinfo}

In this section we consider the case of partial information.
Specifically, we assume that all customers know which queue is being
served by the server, but the individual queue lengths at the two
queues are not known. We call the queue that the server is at the
busy queue and the other queue the idle queue. We assume that if
both queues become empty after a service completion, the server
stays at the queue it served last. Thus the busy queue and idle queue
are well defined at all times.

Now the most general policy that a customer can follow is  described
by a single parameter $p \in [0,1]$ as follows: join the busy queue
with probability $p_1 = p$ and join the idle queue  with probability
$p_2 = 1- p_1$. Hence, under this policy, the Poisson arrival rates in the two queues depend on the server location. This system with ``smart customers'' has been analyzed by Boon et al.~\cite{BW10}.
Let $L_B$ be the expected number of customers in the
busy queue, and $L_I$ be the expected number of customers in the
idle queue, in steady state, under this policy. The next theorem
gives these two quantities, see e.g. Boon et al.~\cite{BW10}.

\begin{theorem}\label{th:parLB}
Under the above policy we have, for all $0 \le p \le 1$,
\[ L_B = \frac{\rho(1-\rho_1)}{(1-\rho_1)^2-\rho_2^2},\]
\[ L_I =\frac{\rho \rho_2}{(1-\rho_1)^2-\rho_2^2}.\]
\end{theorem}

Using the above theorem we derive the socially optimal  policies in
the next theorem.

\begin{theorem}\label{th:parsoc} {\bf (Socially Optimal Policies)}
\begin{enumerate}
\item[i.]
If $c > d$, the socially optimal policy is for everyone to join the idle queue.
\item[ii.]
If $c = d$, all policies are socially optimal.
\item[iii.]
If $c < d$, the socially optimal policy is for everyone to join the busy queue.
\end{enumerate}
\end{theorem}

\begin{proof}
The socially optimal policy minimizes the  expected
cost of a customer in steady state. The expected cost of an arriving
customer in steady state is given by
\[
C = p_1 C_B + p_2 C_I ,
\]
where
$C_B$ and $C_I$ are the expected cost of joining the busy and idle queue, respectively,
\begin{eqnarray*}
C_B &=& \frac1{\mu} \left( c L_B + c \right), \\
C_I &=& \frac1{\mu} \left( d L_B \frac1{1-\rho_1} + c L_I + c \right).
\end{eqnarray*}
This can be simplified to get
\[
C = \frac{c}{\mu (1-\rho)} + \frac{(d-c)\rho_2}{\mu(1-\rho)(1+\rho_2-\rho_1)}.
\]
Using $\rho_2 = \rho-\rho_1$, direct calculations show that
\[
\frac{dC}{d\rho_1} = \frac{c-d}{\mu(1-\rho_1 + \rho_2)^2}.
\]
Thus if $c > d$, $C$ is an increasing function of $\rho_1$,  hence
it is minimized at $\rho_1 = 0$. That is, the socially optimal
policy is for everyone to join the idle queue.  On the other hand,
if $c < d$, $C$ is a decreasing function of $\rho_1$, hence it is
minimized at $\rho_1 = 1$. Then the socially optimal policy is
for everyone to join the busy queue. If $c = d$, then the cost does
not depend on $\rho_1$, and all policies are optimal.
\end{proof}

The next theorem states the results about the individually  optimal
(Nash equilibrium) policies.

\begin{theorem}\label{th:parNash} {\bf (Nash Equilibrium Policies)}
\begin{enumerate}
\item[i.]
If $c(1-\rho) > d$: the Nash equilibrium is policy is to join the idle
queue. It is also a socially optimal policy.
\item[ii.]
If $c(1-\rho) \le d < c:$ the Nash equilibrium policy is to join the
busy queue. It is not the socially optimal policy.
\item[iii.]
If $c \le d:$
the Nash equilibrium policy is to join the busy queue. It is also a
socially optimal policy.
\end{enumerate}
\end{theorem}

\begin{proof}
Suppose arriving customers join the busy queue  with
probability $p_1$ and the idle queue with probability $p_2=1-p_1$.
Now suppose a smart customer knows how the other customers are
behaving, and decides to use this system to minimize her own waiting
costs.
If
she joins the busy queue, her expected cost is $C_B$ and otherwise,
if she joins idle queue, her expected cost is $C_I$.
Using the formulas for $L_B$ and $L_I$ from Theorem~\ref{th:parLB} we get
\begin{equation} \label{eq:CBs}
C_B - C_I = \frac{\rho}{\mu(1-\rho)} \cdot \frac{c(1-\rho)-d}{1-\rho_1+\rho_2}.
\end{equation}
Note that the sign of $C_B-C_I$ does not depend on $p_1$,
the policy followed by all the other customers.
Now consider three cases.
\begin{enumerate}
\item[i.]
$c(1-\rho) > d$: Equation~\eqref{eq:CBs}  implies that $C_B >
C_I$, and hence the smart customer will also join the
idle queue, regardless of what the other customers are doing. Thus
joining the idle queue is a Nash equilibrium. In this case we also
have $c > d$. Hence from Theorem~\ref{th:parsoc}, the socially
optimal policy is to join the idle queue. Thus Nash equilibrium is
also the socially optimal policy.
\item[ii.]
$c(1-\rho) \le d < c:$ In this case, $C_B < C_I$, and hence
the smart customer will join the busy queue, regardless of what the
other customers are doing. Hence joining the busy queue is a Nash
equilibrium policy. However, the socially optimal policy is for
everyone to join the idle queue. Thus the Nash equilibrium is to
join the busy queue, but the socially optimal policy is to join the
idle queue! Individual optimization in this case actually
maximizes the social cost.
\item[iii.]
$c \le d$: The analysis is similar.
\end{enumerate}
\end{proof}

\begin{remark}
{\rm
We can write the condition $c(1-\rho) > d$ as
\[ \frac{c}{\mu} > \frac{d}{\mu-\lambda}.\]
The left hand side is the expected cost of waiting in the busy queue
for one service time, while the right hand side is the expected cost
of waiting in the idle queue for a busy period  initiated by a
single customer. It makes sense that the smart customers weighs
these two costs in order to make her decision, while the social
optimizer compares $c$ and $d$. This results in the Nash equilibrium
policies sending more customers to the busy queue than the socially
optimal policies.
}
\end{remark}

\section{Complete Information}
\label{sec:cominfo}

Now suppose every customer has complete knowledge of the state of
the system, namely, the server location and the length of
each queue. How would the customers use this information to decide
which queue to join?

\subsection{Single Smart Customer}

Suppose that the static routing policy $p$
is under effect (which does not have to be optimal) and that a special
customer wants to use this information to minimize her own expected total
waiting cost. She can observe the number of customers in the two
queues when she arrives at the system: $i$ in the busy queue, and
$j$ in the idle queue. Which queue
should she join?

If she joins the busy queue, her total expected cost is $ci/\mu$. If
she joins the idle queue, the total expected waiting cost is
$di/(\mu-\lambda p) + cj/\mu$. Here the first term represents the
$i$ busy periods that she must wait before the server starts serving
the idle queue (and making it the busy queue). Thus it is optimal for her
to join the queue under service if
\[\frac{ci}{\mu} < \frac{di}{\mu-\lambda p} + \frac{cj}{\mu},\]
and join the idle queue if
\[\frac{ci}{\mu} \ge \frac{di}{\mu-\lambda p} + \frac{cj}{\mu}.\]
Clearly, she could choose either queue if equality holds. If $d=0$, her decision rule
reduces to join the shortest queue. Else, the decision rule
is a linear switching curve.

\subsection{Smart Customer Population}
Now suppose all customers are smart, and each makes a decision to
minimize her own total expected waiting cost, assuming that other
customers will do the same. If $d=0$, each customer will decide to
join the shortest queue, and since this decision is independent of
how the other customers behave, this produces a Nash
equilibrium. The case $d \ge c$ is also obvious:
each customer will decide to join the busy queue, which is a Nash equilibrium.
However, the case $0 < d < c$ is not so obvious. In this case, the single smart customer's decision was made under the assumption that all other
customers join the busy queue with probability $p$ and the idle
queue with probability $1-p$. However, if every customer chooses the
policy derived by the single smart customer, then the single
customer's analysis is no longer valid.
%In the remainder of this section we consider the case $0 < d < c$.\\

So suppose we are given a
decision function $f:\{0,1,2,\cdots\}\times\{0,1,2,\cdots\}
\rightarrow \{I,B\}$ such that $f(i,j) = B$ ($f(i,j) = I$) implies
that an arriving customer that finds $i$ customers in the busy queue
and $j$ customers in the idle queue joins the busy (idle)  queue.
Let $\tau_f(i,j)$ be the expected time until the busy queue empties
if the system starts with $i$ customers in the busy queue, and $j$
in the idle queue, under decision function $f$. Note that $\tau_f (i,j)$ is bounded by $i / (\mu - \lambda)$, which is the expected time to empty the busy queue if all future arrivals are sent to the busy queue.
It is individually optimal to join the busy queue if
\[\frac{ci}{\mu} < d\tau_f(i,j+1) + \frac{cj}{\mu},\]
and to join the idle queue if
\[\frac{ci}{\mu} \ge d\tau_f(i,j+1) + \frac{cj}{\mu}.\]
We say that  $f^*$ is an individually optimal decision function if
\[ f^*(i,j) = B \Leftrightarrow \frac{ci}{\mu} < d\tau_{f^*}(i,j+1) + \frac{cj}{\mu},\]
and
\[ f^*(i,j) = I \Leftrightarrow \frac{ci}{\mu} \ge d\tau_{f^*}(i,j+1) + \frac{cj}{\mu}.\]
The function $f^*$ also describes a Nash
equilibrium policy.

We now present a recursive method to compute $f^*$.
We consider a finite horizon system that operates as follows. Let $n \ge 0$ be a given integer (the ``horizon''). Let $(i,j)$ be the initial state of the system ($i \ge 1, j \ge 0$). We assume that after $n$ events (arrivals or departures), arrivals are turned off and only departures are allowed to occur, and the systems stops operation once it becomes empty. Let $\delta_n(i,j)$ represent the new state of the system if a customer arrives when the horizon is $n$, and the system is in state $(i,j)$ and chooses an
action that minimizes her own cost. Let  $\tau_n(i,j)$ be the expected
time until the busy queue becomes empty if the system with horizon $n$ starts in
state $(i,j)$, and all the arrivals behave in an individually optimal way.
We have
\begin{equation}
\tau_0(i,j)  =  \frac{i}{\mu}, \quad i \ge 1, j \ge 0.
\end{equation}
This reflects that a zero horizon system has no more arrivals and hence the server completes the work in the current queue after an expected time of $i/\mu$.  Now recursively define for all $n \ge 0, i \ge 1, j \ge 0$,
\begin{eqnarray*}
\delta_{n}(i,j) & = & \left\{\begin{array}{ll}
(i+1,j) & \;\; \mbox{if} \; ci/\mu < d\tau_n(i,j+1) + cj/\mu,\\
(i,j+1) & \;\; \mbox{if} \; ci/\mu \ge d\tau_n(i,j+1) + cj/\mu,
\end{array}\right. %\quad i \ge 1, j \ge 0,
\\
%\tau_{n+1}(1,j) & = & \frac{1}{\lambda+\mu} + \frac{\lambda}{\lambda+\mu}\tau_n(\delta_n(1,j)), \quad j \ge 0,\\
\tau_{n+1}(i,j) & = & \frac{1}{\lambda+\mu} + \frac{\mu}{\lambda+\mu}\tau_n(i-1,j)
+ \frac{\lambda}{\lambda+\mu}\tau_n(\delta_n(i,j)), %\quad i \ge 1, j \ge 0,
\end{eqnarray*}
where $\tau_n (0,j) = 0$. The next lemma formulates monotonicity properties of $\tau_n (i,j)$.

\begin{lemma}\label{lem:mono}
For all $n \ge 0, i \ge 1, j \ge 0$,
\begin{eqnarray}
  \tau_n (i,j)&\le& \tau_n (i,j+1), \label{eq:mon2}\\
  \tau_n (i,j+1) &\le& \tau_n (i+1,j), \label{eq:mon1} \\
  \tau_n (i,j) & \le & \frac{i}{\mu-\lambda} , \label{eq:mon2b}\\
  \tau_{n} (i,j) &\le& \tau_{n+1} (i,j). \label{eq:mon3}
\end{eqnarray}
\end{lemma}

\begin{proof}
By induction. For $n=0$ we have
\[
\tau_0 (i,j) = \tau_0 (i,j+1) = \frac{i}{\mu} < \frac{i+1}{\mu} = \tau_0 (i+1,j)
\]
and
\[
\tau_1 (i,j) \ge \frac{1}{\lambda+\mu} + \frac{\mu}{\lambda+\mu}\frac{i-1}{\mu}
+ \frac{\lambda}{\lambda+\mu}\frac{i}{\mu}  = \frac{i}{\mu} = \tau_0 (i,j).
\]
Hence, (\ref{eq:mon2})-(\ref{eq:mon3}) hold for $n=0$. Now assume (\ref{eq:mon2})-(\ref{eq:mon3}) hold for $n$. Then we will show that these inequalities also hold for $n+1$.
To establish (\ref{eq:mon2}) for $n+1$, consider
\begin{eqnarray*}
\tau_{n+1} (i,j+1) - \tau_{n+1} (i,j) & = & \frac{\mu}{\lambda+\mu} (\tau_n(i-1,j+1)-\tau_n(i-1,j)) \\
&+& \frac{\lambda}{\lambda+\mu}(\tau_n(\delta_n(i,j+1)-\tau_n(\delta_n(i,j)) .
\end{eqnarray*}
The first term is nonnegative by (\ref{eq:mon2}). If $\delta_n (i,j+1) = (i+1,j+1)$, then for both $\delta_n (i,j) = (i+1,j)$ and $\delta_n (i,j) = (i,j+1)$, we can conclude that the second term is nonnegative by application of (\ref{eq:mon1}) and (\ref{eq:mon2}). If $\delta_n (i,j+1) = (i,j+2)$, then $\delta_n (i,j) = (i,j+1)$ by (\ref{eq:mon2}), and thus we can again conclude that the second term is nonnegative by (\ref{eq:mon2}).
For (\ref{eq:mon1}) we get
\begin{eqnarray*}
\tau_{n+1} (i+1,j) - \tau_{n+1} (i,j+1) & = & \frac{\mu}{\lambda+\mu} (\tau_n(i,j)-\tau_n(i-1,j+1)) \\
&+& \frac{\lambda}{\lambda+\mu}(\tau_n(\delta_n(i+1,j)-\tau_n(\delta_n(i,j+1)) .
\end{eqnarray*}
The first term on the right-hand side is nonnegative by (\ref{eq:mon1}). If $\delta_n (i+1,j) = (i+2,j)$, then for both $\delta_n (i,j+1) = (i+1,j+1)$ and $\delta_n (i,j+1) = (i,j+2)$, we obtain that the second term is nonnegative by (repeated) application of (\ref{eq:mon1}). If $\delta_n (i+1,j) = (i+1,j+1)$, we come to the same conclusion.
By (\ref{eq:mon2}),
\begin{eqnarray*}
\tau_{n+1} (i,j) %&=& \frac{1}{\lambda+\mu} + \frac{\mu}{\lambda+\mu} \tau_n (i-1,j) + \frac{\lambda}{\lambda+\mu} \tau_n (\delta_n(i,j)) \\
&\le& \frac{1}{\lambda+\mu} + \frac{\mu}{\lambda+\mu} \tau_n (i-1,j)
+ \frac{\lambda}{\lambda+\mu} \tau_n (i+1,j),
\end{eqnarray*}
and thus by (\ref{eq:mon2b}),
\[
\tau_{n+1} (i,j) \le \frac{1}{\lambda+\mu} + \frac{\mu}{\lambda+\mu} \frac{i-1}{\mu-\lambda}
+ \frac{\lambda}{\lambda+\mu} \frac{i+1}{\mu-\lambda} = \frac{i}{\mu-\lambda}.
\]
Finally, to prove (\ref{eq:mon3}) for $n+1$,
\begin{eqnarray*}
\tau_{n+2} (i,j) - \tau_{n+1} (i,j) & = & \frac{\mu}{\lambda+\mu} (\tau_{n+1}(i-1,j)-\tau_{n}(i-1,j)) \\
&+& \frac{\lambda}{\lambda+\mu}(\tau_{n+1}(\delta_{n+1}(i,j)-\tau_n(\delta_n(i,j)) .
\end{eqnarray*}
The first term is nonnegative by (\ref{eq:mon3}). If $\delta_{n+1} (i,j) = (i+1,j)$, then for both $\delta_n (i,j) = (i+1,j)$ and $\delta_n (i,j) = (i,j+1)$, it follows that the second term is nonnegative by application of (\ref{eq:mon1}) for $n+1$ and (\ref{eq:mon3}). If $\delta_{n+1} (i,j) = (i,j+1)$, then also $\delta_n (i,j) = (i,j+1)$ by (\ref{eq:mon2}), and thus the second term is nonnegative by (\ref{eq:mon3}).
\end{proof}

The following theorem states that this recursive procedure generates an individually optimal decision function $f^*$.

\begin{theorem}\label{th:fulNash}
For all $i \ge 1, j \ge 0$,
\[
\lim_{n \rightarrow \infty} \tau_n(i,j) = \tau (i,j) = \tau_{f^*} (i,j)
\]
and
\[
\lim_{n \rightarrow \infty} \delta_n(i,j) = \delta(i,j),
\]
where $f^*$ is defined as
\begin{eqnarray*}
f^*(i,j) = B  &\Leftrightarrow& \delta(i,j) = (i+1,j),\\
f^*(i,j) = I  &\Leftrightarrow&  \delta(i,j) = (i,j+1).
\end{eqnarray*}
\end{theorem}

\begin{proof}
By virtue of (\ref{eq:mon2b})-(\ref{eq:mon3}), the sequence $\tau_n(i,j)$ is non-decreasing in $n$ and bounded. Hence the limits of $\tau_n(i,j)$ and $\delta_n (i,j)$ exist and satisfy for all $i \ge 1, j \ge 0$,
\begin{eqnarray*}
\delta_(i,j) & = & \left\{\begin{array}{ll}
(i+1,j) & \;\; \mbox{if} \; ci/\mu < d\tau(i,j+1) + cj/\mu,\\
(i,j+1) & \;\; \mbox{if} \; ci/\mu \ge d\tau(i,j+1) + cj/\mu,
\end{array}\right.\\
%\tau(1,j) & = & \frac{1}{\lambda+\mu} + \frac{\lambda}{\lambda+\mu}\tau(\delta(1,j)), \quad j \ge 0,\\
\tau(i,j) & = & \frac{1}{\lambda+\mu} + \frac{\mu}{\lambda+\mu}\tau(i-1,j)
+ \frac{\lambda}{\lambda+\mu}\tau(\delta(i,j)),
\end{eqnarray*}
where $\tau(0,j) = 0$. The expected values $\tau_{f^*} (i,j)$ satisfy for all $i \ge 1, j \ge 0$,
\begin{eqnarray*}
%\tau(1,j) & = & \frac{1}{\lambda+\mu} + \frac{\lambda}{\lambda+\mu}\tau(\delta(1,j)), \quad j \ge 0,\\
\tau_{f^*}(i,j) & = & \frac{1}{\lambda+\mu} + \frac{\mu}{\lambda+\mu}\tau_{f^*}(i-1,j)
+ \frac{\lambda}{\lambda+\mu}\tau_{f^*}(\delta(i,j)), % \quad i \ge 1, j \ge 0,
\end{eqnarray*}
where $\tau_{f^*} (0,j) = 0$. To prove $\tau_{f^*} (i,j) = \tau(i,j)$, consider $v(i,j) = \tau_{f^*} (i,j) - \tau (i,j)$  satisfying
\begin{eqnarray*}
%\tau(1,j) & = & \frac{1}{\lambda+\mu} + \frac{\lambda}{\lambda+\mu}\tau(\delta(1,j)), \quad j \ge 0,\\
v(i,j) & = & \frac{\mu}{\lambda+\mu} v(i-1,j)
+ \frac{\lambda}{\lambda+\mu}v(\delta(i,j)), \quad i \ge 1, j \ge 0,
\end{eqnarray*}
or in vector-matrix notation
\begin{equation}\label{eq:vector1}
  v = P v,
\end{equation}
where $P$ is the (transient) transition probability matrix with
\[
P_{(i,j), (i-1,j)} = 1 - P_{(i,j), \delta(i,j)} = \frac{\mu}{\lambda+\mu}, \quad i \ge 1, j \ge 0.
\]
Iterating (\ref{eq:vector1}) yields $v = P^n v$. Since transitions are restricted to neighboring states, $P^n_{(i,j), (k,l)} = 0$ for all $(k,l)$ with $k > i+n$. Hence, since $\tau (i,j)$ and $\tau_{f^*} (i,j)$ are bounded by
$i / (\mu -\lambda)$,
\begin{equation}\label{eq:vector2}
|v(i,j)| = \left| \left( P^n v \right)_{(i,j)} \right| \leq \left( P^n \boldsymbol{1} \right)_{(i,j)} \frac{2(n+i)}{\mu-\lambda} ,
\end{equation}
where $\boldsymbol{1}$ is the all-one vector. $\left( P^n \boldsymbol{1} \right)_{(i,j)}$ is the probability that the Markov chain $P$ does not reach the absorbing boundary $i=0$ in $n$ transitions, when starting in $(i,j)$. This probability is bounded by $P(X_i > n)$, where $X_i$ is the number of transitions to reach $0$ of the random walk on the non-negative integers with one-step probabilities $P_{j,j-1} = 1 - P_{j,j+1} = \frac{\mu}{\lambda+\mu}$, when it starts in state $i$. This random walk reflects that all future arrivals are sent to the busy queue. By Markov's inequality, $P(X_i > n) \le E(X_i^2) / n^2$. Hence, from (\ref{eq:vector2}),
\[
|v(i,j)| \leq P(X_i > n) \cdot \frac{2(n+i)}{\mu-\lambda} \leq \frac{E(X_i^2)}{n^2} \cdot \frac{2(n+i)}{\mu-\lambda}.
\]
Letting $n \rightarrow \infty$, we conclude that $v(i,j) = 0$, which completes the proof.
\end{proof}

Next we describe the main structural properties of the policy $f^*$.

\begin{theorem}\label{theorem:BusyForSure}
$f^*(i,j) = B$ for all $1 \leq i \leq j$.
\end{theorem}

\begin{proof}
For $1 \leq i \leq j$,
\begin{align*}
    \frac{ci}{\mu} < d \tau(i,j+1) + \frac{cj}{\mu},
\end{align*}
since $\tau(i,j+1) > 0$. Hence, $f^* (i,j) = B$ by definition.
\end{proof}

The above theorem  says that if the busy queue is no longer than the idle queue, then the individually optimal decision for any customer is to join the busy queue. The theorem below states monotonicity of the individually optimal policy in $j$.

\begin{theorem}\label{theorem:BusyJ++}
For all $i \geq 1$, $j \geq 0$, if $f^*(i,j) = B$, then $f^*(i,j+1) = B$.
\end{theorem}

\begin{proof}
Suppose $f^*(i,j) = B$ for some $i \geq 1$, $j \geq 0$. This implies
\begin{eqnarray*}
    \frac{ci}{\mu} & < & d \tau(i,j+1) + \frac{cj}{\mu}\\
                   & < & d \tau(i,j+1) + \frac{c(j+1)}{\mu} \\
                   & \leq & d \tau(i,j+2) + \frac{c(j+1)}{\mu},
\end{eqnarray*}
where the last inequality follows from (\ref{eq:mon2}) by taking $n \rightarrow \infty$.
Hence, $f^*(i,j+1) = B$.
\end{proof}

To prove monotonicity in $i$, we first need a technical result.

\begin{theorem}\label{theorem:Concavity}
Suppose $f^*(k,j) = B$ for every $1 \le k \leq i$ and fixed $j \geq
0$. Then $\tau(k,j)$ is concave
for $1 \leq k \leq i$.
\end{theorem}

\begin{proof}
Fix $j \ge 0$. First, we show by induction that for all $1 \le k \le i-1$,
\begin{align}
    \tau(k+1,j) - \tau(k,j) \leq \frac{1}{\mu-\lambda} . \label{inequality:TauUpperbound}
\end{align}
For $k = 1$,
\begin{eqnarray*}
    \tau(2,j) - \tau(1,j) & = &\tau(2,j) - \frac{1}{\lambda+\mu} - \frac{\lambda}{\lambda+\mu} \ \tau(2,j)\\
                          & = &\frac{\mu}{\lambda+\mu} \ \tau(2,j) - \frac{1}{\lambda+\mu} \\
                          & \leq &\frac{\mu}{\lambda+\mu} \ \frac{2}{\mu-\lambda} - \frac{1}{\lambda+\mu} \; = \; \frac{1}{\mu-\lambda},
\end{eqnarray*}
where the inequality follows from the bound $\tau(2,j) \le 2 / (\mu - \lambda)$.
Now we assume that (\ref{inequality:TauUpperbound}) holds for $k \le i-2$ and then show that it also holds for $k+1$.
\begin{eqnarray*}
    \tau(k+2,j) - \tau(k+1,j) & = &\tau(k+2,j) - \frac{1}{\lambda+\mu} - \frac{\mu}{\lambda+\mu}\tau(k,j) - \frac{\lambda}{\lambda+\mu}\tau(k+2,j) \\
                              & = & \frac{\mu}{\lambda+\mu} \ [\tau(k+2,j)-\tau(k,j)] - \frac{1}{\lambda+\mu} \\
                              & = & \frac{\mu}{\lambda+\mu} \ [\tau(k+2,j)-\tau(k+1,j)] + \frac{\mu}{\lambda+\mu} \ [\tau(k+1,j)-\tau(k,j)]\\
                              &   &  - \frac{1}{\lambda+\mu} .
\end{eqnarray*}
Hence
\begin{eqnarray*}
    \tau(k+2,j) - \tau(k+1,j) & = & \frac{\mu}{\lambda} \ [\tau(k+1,j)-\tau(k,j)] - \frac{1}{\lambda} \\
                              & \leq & \frac{\mu}{\lambda(\mu-\lambda)} -  \frac{1}{\lambda} \; = \; \frac{1}{\mu-\lambda} ,
\end{eqnarray*}
which concludes the proof of (\ref{inequality:TauUpperbound}).
Next, to establish concavity, we have for $1 \leq k \leq i-2$,
\begin{align*}
    & [\tau(k+2,j)-\tau(k+1,j)] - [\tau(k+1,j)-\tau(k,j)]\\
    & \quad = \; \tau(k+2,j) - 2 \left[\frac{1}{\lambda+\mu} + \frac{\mu}{\lambda+\mu}\ \tau(k,j) + \frac{\lambda}{\lambda+\mu} \ \tau(k+2,j)\right] + \tau(k,j) \\
    & \quad = \; \frac{\mu-\lambda}{\mu + \lambda} \ [\tau(k+2,j)-\tau(k,j)] - \frac{2}{\lambda+\mu} \\
    & \quad \leq \; \frac{\mu-\lambda}{\mu + \lambda} \ \frac{2}{\mu-\lambda} - \frac{2}{\lambda+\mu} \; = \; 0 ,
\end{align*}
where the inequality follows by repeated application of (\ref{inequality:TauUpperbound}).
\end{proof}

With the above result we can prove monotonicity in $i$.

\begin{theorem}\label{theorem:BusyI--}
For all $i \geq 2$, $j \geq 0$, if $f^*(i,j) = B$, then $f^*(i-1,j) = B$.
\end{theorem}

\begin{proof}
Fix $i \ge 2$. By downward induction we will prove for $j \ge 0$ that $f^* (i,j) = B$ implies $f^*(k,j) = B$ for all $1 \le k \le i$. By Theorem \ref{theorem:BusyForSure}
this is true for $j \ge i$. Now we assume that it holds for $j$ and then show that it also holds for $j-1$. Suppose $f^*(i,j-1) = B$. If $j > 1$, then $f^* (1,j-1) = B$ by Theorem \ref{theorem:BusyForSure}. To show that this is also valid for $j=1$, first note that $f^*(i,0) = B$ implies
\[
\frac{c i}{\mu} < d \tau(i,1) ,
\]
and thus, by using $\tau(i,1) \le i / (\mu - \lambda)$,
\[
\frac{d}{\mu - \lambda} > \frac{c }{\mu} .
\]
Hence,
\[
d \tau(1,1) = d \left[ \frac{1}{\lambda+\mu} + \frac{\lambda}{\lambda+\mu} \ \tau(2,1) \right] \ge d \left[ \frac{1}{\lambda+\mu} + \frac{\lambda}{\lambda+\mu} \ \frac{2}{\mu-\lambda} \right] = \frac{d}{\mu - \lambda} > \frac{c}{\mu},
\]
so $f^*(1,0) = B$. Since $f^*(1,j-1) = f^*(i,j-1) = B$, we have
\[
\frac{c}{\mu} < d \tau(1,j) + \frac{c(j-1)}{\mu}, \quad \frac{ci}{\mu} < d \tau(i,j) + \frac{c(j-1)}{\mu},
\]
and thus for $1 \le k \le i$,
\begin{eqnarray*}
  \frac{ck}{\mu} &=& \frac{i-k}{i-1} \ \frac{c}{\mu} + \frac{k-1}{i-1} \ \frac{ci}{\mu} \\
   & <   &  d \left[ \frac{i-k}{i-1} \ \tau(1,j) + \frac{k-1}{i-1} \ \tau(i,j) \right] + \frac{c(j-1)}{\mu} \\
   & \le &  d \tau(k,j) + \frac{c(j-1)}{\mu},
\end{eqnarray*}
where the second inequality follows from Theorem \ref{theorem:Concavity}. Hence $f^*(k,j-1) = B$.
\end{proof}

Theorems \ref{theorem:BusyJ++} and \ref{theorem:BusyI--} imply that the individually optimal policy is described by a switching curve $h(\cdot)$ such that it is optimal for a customer to join the busy queue in state $(i,j)$ is $j >  h(i)$, and that $h$ is a non-decreasing function of $i$. Note that $h$ depends on the costs $c$ and $d$. This completes the discussion of the individually optimal policy.

\subsection{Socially Optimal Policy}

Finally, suppose there is a
central controller who can route the customers so as to minimize the
long run expected waiting cost per unit time.
Let $Z(t)$ be the total number of customers in the system (those in the busy queue plus those in the idle queue) at time $t$. We begin with an easy but important observation.

\begin{lemma}
$\{Z(t), t \ge 0\}$ is the queue length process in an $M/M/1$ queue regardless of the routing policy followed.
\end{lemma}
\begin{proof}
The total arrival process to the system is a Poisson process with rate $\lambda$, the service times are independent and exponential with rate $\mu$, and the polling service discipline is work conserving. Hence the lemma follows.
\end{proof}

Now let $X(t)$ be the number of customers in the busy queue and $Y(t)$ be the number in the idle queue at time $t$. Then the total cost $C(t)$ over $(0,t]$ is given by
\[ C(t) = \int_0^t (c X(u) + d Y(u))\text{d}u, \;\;\; t \ge 0.\]
The process $\{C(t), t \ge 0\}$  does depend on the routing policy. Let $T_n$ be the $n$th time when the system busy cycle ends, i.e., when $Z(t)$ reaches $0$. Let
 \[ C_n = \int_{T_n}^{T_{n+1}} (c X(u) + d Y(u))du\]
be the total cost incurred over the interval $(T_n, T_{n+1}]$. An important implication of the above lemma is that $\{C(t), t \ge 0\}$ is a (delayed in case $Z(0)>0$) renewal reward process, since $\{(C_n, T_{n+1}-T_n), n \ge 1\}$ is a sequence of independent and identically distributed bi-variate random variables. Furthermore, $\{T_{n+1}-T_n, n \ge 1\}$ is a sequence of independent busy cycles in an $M/M/1$ queue. Hence, their common distribution does not depend on the routing policy, and
\[ E(T_{n+1}-T_n) = \frac{\mu}{\lambda(\mu - \lambda)} < \infty.\]
Then, from the results on renewal reward processes, we obtain that
\[ \lim_{t \rightarrow \infty} \frac{C(t)}{t} = \frac{\lambda(\mu-\lambda)}{\mu} E\left(\int_0^{T_1} C(u) \text{d}u \; | \; Z(0)=1\right).\]
Also, we have the following bound:
\[ \lim_{t \rightarrow \infty} \frac{C(t)}{t} \le \max(c,d)\lim_{t \rightarrow \infty} E(Z(t)) = \max(c,d)\frac{\lambda}{\mu - \lambda} < \infty.\]
Thus the long run average cost exists and is finite, and it is proportional to the total cost in first busy cycle started with one customer in the system. Thus the problem of finding an average cost optimal policy reduces to the problem of finding an optimal policy that minimizes the total expected cost $C$ over a busy period $T$ starting in state (1,0). This can be written as
\[
C = E\left(\int_0^T (cX(t) + dY(t))dt \right)= d E\left(\int_0^T Z(t)dt\right) + (c-d)E\left(\int_0^T X(t) dt\right).
\]
Clearly the first term is independent of the routing policy followed. Thus, to minimize $C$, we need to minimize the integral $\int X(t)dt$ if  $ c > d$ and maximize it if $c  < d$. If $c = d$, any policy is optimal. Clearly, when $ c \le d$, it is optimal to send all traffic to the busy queue. The interesting case arises when $c > d$. Hence we deal with that case below.

The above discussion implies that, without loss of generality, we can assume $c=1$, and $d=0$. Note that this is in stark contrast with the individually optimal policies that depend on both $c$ and $d$. We can now formulate the cost minimization as a standard negative dynamic programming problem, see e.g. Ross~\cite{R83}. Below we make the details precise.

Let $v(i,j)$ be the minimum expected total cost starting in state $(X(0),Y(0))=(i,j)$ over the time interval $[0,T)$ where
\[ T = \min\{t \ge 0: Z(t) = 0 \; | \; Z(0)=i+j\}.\]
Without loss of generality we can assume that $\lambda + \mu =1$. From Ross~\cite{R83} it follows that $v$ satisfies the optimality equations:
\begin{eqnarray*}
    v(i,j) & = & i  + \mu v(i-1,j) + \lambda \min (v(i+1,j),v(i,j+1)), \quad i \geq 2, ~j \ge 0, \\
    v(1,j) & = & 1 + \mu v(j,0) + \lambda \min (v(2,j),v(1,j+1)), \quad j \ge 0,
\end{eqnarray*}
where $v(0,0) = 0$.
We are interested in the solution to the above equations, which can be obtained by the following value iteration
for $n \ge0$:
\begin{eqnarray*}
v_{n+1}(i,j) &=& i  + \mu v_{n}(i-1,j) + \lambda \min (v_{n}(i+1,j),v_{n}(i,j+1)) \quad i \geq 2, ~j \geq 0, \\
    v_{n+1}(1,j) &=& 1 + \mu v_{n}(j,0) + \lambda \min (v_{n}(2,j),v_{n}(1,j+1)) \quad j \geq 0,
\end{eqnarray*}
with initially $v_0 (i,j) = 0$ for all $i \ge 1, j \ge 0$ and $v_n (0,0) = 0$ for all $n \ge 0$.

Note that the $v_n$ in the above iteration is guaranteed to converge to $v$ as $n \rightarrow \infty$, even though the costs are unbounded. Once $v$ is computed, the theory of negative dynamic programming says that the  optimal policy in state $(i,j)$ is to route an incoming customer to the busy queue if $v(i+1,j) < v(i, j+1)$, and to the idle queue otherwise.

Unfortunately, we have been unable to formally derive any structural results for socially optimal policy, the main stumbling block being the term $v_n(j,0)$ on the right hand side of the equation for $v_{n+1}(1,j)$. However, based on extensive numerical experimentation we have seen that a switching curve policy is optimal. That is, for each $i \ge 1$, there is a critical number $g(i)$ such that the optimal policy in state $(i,j)$ is to route the incoming customer to the busy queue if $j > g(i)$ and to the idle queue otherwise. Furthermore, in a fairly large parameter space, the switching curve can be approximated by the linear function:
\begin{equation} \label{eq:con}
 g(i) = \alpha i, \;\;\; i \ge 0.
 \end{equation}
where
\begin{equation} \label{eq:alp2}
\alpha =  \frac{2\rho}{-1+\rho+\sqrt{(1-\rho)(1+3\rho)}}.
\end{equation}
It is easy to see that $\alpha > 1$. We have also observed numerically that $h(i) \le g(i)$ for all $i \ge 1$, where $h$ is the switching curve for the individually optimal policy. That is, more customers join the busy queue under the individually optimal policy than under the socially optimal policy. We shall illustrate these comments with a numerical example in Section~\ref{sec:num}.

In the next section we shall develop a fluid model of this scenario and derive an optimal routing policy.

\section{Fluid Model}
\label{sec:fluid}

Consider a fluid equivalent of the polling system described in Section~\ref{sec:mod}. Customers arrive as a fluid with deterministic rate $\lambda$ per unit time and can be routed to the busy queue or the idle queue. The cost structure remains the same.  Once the server empties a queue, he switches to the other queue and continues to empty it. The fluid is removed at a deterministic rate $\mu > \lambda$ as long as there is fluid to be removed. Once the system becomes empty, the fluid is removed at rate $\lambda$, and the system stays empty forever. As before, the queue that is being served is called the busy queue, and the other one is called the idle queue.

Let $x(t)$ be the amount of fluid in the busy queue, and $y(t)$ the amount of fluid in the idle queue, at time $t$. Suppose the initial state is $x(0) = x_0 \ge 0$ and $y(0) = y_0 \ge 0$. Let $z(t) = x(t) + y(t)$ be the total fluid in the system at time $t$. Then $z(0)=x(0)+y(0)$ and regardless of the routing policy followed, $z(t)$ decreases at rate $\lambda - \mu< 0$, until it hits zero at time
\[ T = \frac{x_0+y_0}{\mu - \lambda},\]
and then $z(t)$ remains 0 for $t \ge T$. As in the previous section, the total cost incurred can be written as
\[ \int_0^T (c x(t) + d y(t)) \mbox{d}t = d\int_0^T z(t) dt + (c-d)\int_0^T x(t)\mbox{d}t.\]
We want to determine the optimal routing policy for the incoming fluid so as to minimize this cost. Since the routing policy does not affect the trajectory of $z$, the optimal policy  needs to minimize the second integral if $c > d$, and maximize it if $c < d$.

The optimal routing policy for $t \ge T$ is obvious:  keep sending the incoming fluid to the busy queue, and  both the queues will remain empty forever. Thus we  concentrate on the optimal policy for $0 \le t \le T$. If $c < d$, the optimal policy is to route all traffic to the busy queue. If $c = d$ all routing policies are optimal. Hence we further concentrate on the case $c > d$ in the rest of this section.

We assume the server has just switched to queue 1 at time zero and the system is non-empty. Thus $x_0 > 0$ and $y_0 = 0$, and  queue 1 is the busy queue at time zero. Since the system is entirely deterministic, the routing policy is completely described by the instantaneous rate at which the incoming fluid is routed to the two queues as a function of time. Let $r(t)$ be the rate at which fluid is routed to the idle queue at time $t$, $t \ge 0$. Then $\lambda - r(t)$ is the rate at which fluid is routed to the busy queue. Now let $t(0)=0$ and $t_k$ be the time at which the server switches from one queue to the other, called the $k$th switching time. These times are completely determined by the function $\{r(t), t \ge 0\}$ as follows:
\[ t_1 = \min\{t \ge 0: x(t) = x_0 +\int_0^t (\lambda - r(u))\mbox{d}u - \mu t = 0\}.\]
Thus queue 2 becomes the busy queue at time $t_1$ and now has
\[ x_1 = x(t_1)= \int_0^{t_1} r(u) \mbox{d}u\]
amount of fluid in it. Queue 1 becomes the idle queue and has no fluid in it. Thus we can recursively obtain, for $k \ge 1$:
\begin{equation} \label{eq:tk}
 t_{k+1} = \min\{t \ge t_k: x(t) = x_k +\int_{t_k}^t (\lambda - r(u))\mbox{d}u - \mu (t-t_k) = 0\},
 \end{equation}
\begin{equation} \label{eq:xk}
 x_{k+1} = x(t_{k+1})=  \int_{t_k}^{t_{k+1}} r(u) du.
 \end{equation}
We call $[t_k, t_{k+1})$ the $k$th cycle. Note that $x_{k+1}$ also represents the total amount of fluid routed to the idle queue during the $k$th cycle. The next theorem gives an important preliminary  result on the optimal policy.
\begin{theorem} \label{th:pre}
Let $\{r(t), t \ge 0\}$ be a given routing policy where $r(t)$ is the instantaneous rate at which incoming fluid is routed to the idle queue at time $t$. Let $\{t_k, k \ge 0\}$ and $\{x_k, k \ge 0\}$ be as given in \eqref{eq:tk} and \eqref{eq:xk}. Let
\[ v_k = \frac{x_{k+1}}{\lambda}, \;\;\; k \ge 0\]
and define a new routing policy $\{s(t), t \ge 0\}$ as follows:
\[ s(t) = \left\{ \begin{array}{ll}
\lambda & \;\;\; \mbox{for} \; t_k \le t \le t_k + v_k,\\
0 & \;\;\; \mbox{for} \; t_k + v_k \le t < t_{k+1}.
\end{array} \right.
\]
Then the total cost incurred by routing policy $\{s(t), t \ge 0\}$ is no greater than that of $\{r(t), t \ge 0\}$.
\end{theorem}

\begin{proof}
Let $x(t)$ and $y(t)$ be the fluid levels at time $t$ under the routing policy $\{r(t), t \ge 0\}$, and $x^s(t)$ and $y^s(t)$ be the fluid levels at time $t$ under the routing policy $\{s(t), t \ge 0\}$. First note that the amount of fluid routed to the idle queue under the $\{s(t), t \ge 0\}$ policy during the $k$th cycle is $\lambda v_k = x_{k+1}$, which is the same as under the  $\{r(t), t \ge 0\}$ policy. However, this fluid is routed at the fastest rate possible, namely $\lambda$. Hence
\[ y^s(t) \ge y(t), \;\;\; t_k \le t < t_{k+1}, \;\; k \ge 0.\]
Since $x(t) + y(t) = x^s(t) + y^s(t)$ for all $t$ (since the total fluid content is independent of the routing policy), it follows that
\[ x^s(t) \le x(t), \;\;\; t_k \le t < t_{k+1}, \;\; k \ge 0.\]
In fact, we have
\[ x^s(t) = \left\{ \begin{array}{ll}
x_k - \mu t & \;\;\; \mbox{for} \; t_k \le t \le t_k + v_k,\\
x_k - \mu v_k - (\mu - \lambda)(t-v_k) & \;\;\; \mbox{for} \; t_k + v_k \le t < t_{k+1}.
\end{array} \right.
\]
Thus $x^s(t) > 0$ for $t_k \le t < t_{k+1}$, and $\{x^s(t), t \ge t_k\}$ reaches zero for the first time at time $t_{k+1}$. Thus the switching times  under  routing policy $s$ are the same as under  policy $r$, and $x^s_k = x_k$ for all $k \ge 0$. Thus we have
\begin{eqnarray*}
 \int_0^{t_1} (cx^s(t) + dy^s(t))dt  & = &  \int_0^{t_1} (cx^s(t) + d(z(t) - x^s(t)))dt\\
  & = &  d\int_0^{t_1} z(t) dt +  (c - d) \int_0^{t_1}x^s(t)dt\\
 &\le &  d\int_0^{t_1} z(t) dt +  (c - d) \int_0^{t_1}x(t)dt\\
 & = &   \int_0^{t_1} (cx(t) + dy(t))dt.
 \end{eqnarray*}
Thus the cost under $s$ is no greater than that under $r$ over the first cycle. Since the state of the polling system under both  policies is the same at time $t_1$, the above argument can be repeated to show that  policy $s$ performs at least as well as policy $r$ over every cycle, and hence for all $0 \le t \le T$.
\end{proof}

The above theorem implies that a policy is equally well characterized by the switching times $\{t_k, k \ge 0\}$ it induces (with $t_0=0$), and among all the policies with these switching times, a policy that sends all the traffic to the idle queue first as long as possible in each cycle, is optimal. Thus all that remains to be done is to identify the optimal switching times.

%First observe that if one knows at which time instances the optimal policy
%switches, or if the total number of jobs served while serving a queue is given
%from the optimal policy, then the queue-assignment problem becomes almost
%trivial: first send all jobs to the idle queue, then switch and send all jobs
%to the queue currently in service. This can be seen as follows. Assume that it
%would be optimal to first send jobs to the idle queue, say for a duration T1,
%then to the active queue, say for a duration T2 and then again to the idle
%queue, say for a duration T3. Then an alternative to this policy would be to
%first send jobs to the idle queue for a duration of T1+T3 and then to the
%active queue for a duration T2. This alternative policy is feasible, and
%furthermore, the buffer contents of the active queue is never higher than for
%the assumed to be optimal policy. So along the same lines you can prove by
%contradiction that it is optimal to first send jobs to the idle queue, then
%switch once and send jobs to the active queue until that queue has been
%emptied.
%
%\vskip\baselineskip
%
%So it only remains to determine the switching times for the optimal policy, or
%the amount of jobs served.

Since the system is deterministic, determining optimal $\{t_k, k \ge 0\}$ is equivalent to determining the optimal fluid levels $\{x_k, k \ge 1\}$, with a given initial level $x_0$. The next theorem shows that this can be modeled and solved as a Linear Quadratic Regulator (LQR) problem, see Bryson and Ho~\cite{BH75}, and Bertsekas~\cite{B05}.

\begin{theorem}\label{th:lqr}
The optimal $\{x_k, k \ge 1\}$ are obtained by solving the following infinite horizon constrained LQR:
\begin{subequations}
\label{eq:lqrq2}
\begin{align}
\min_{u_k} & \sum_{k=0}^\infty (1-\rho)x_k^2+\rho u_k^2
\intertext{subject to}
x_{k+1}&=\rho (x_k - u_k) ,\\
0&\leq u_k\leq x_k .
\label{eq:constraint}
\end{align}
\end{subequations}
\end{theorem}

\begin{proof}
Let $\{t_k, k \ge 0\}$ be the switching times of the optimal policy. From Theorem~\ref{th:pre} we see that there exist $\{v_k, k \ge 0\}$ such that over the $k$th cycle $[t_k, t_{k+1})$ it is optimal to route all  fluid to
the idle queue over $[t_k, t_k + v_k)$ and then route all  fluid to the busy queue over $[t_k +  v_k,t_{k+1})$.

Thus, during the interval $[t_k,t_k + v_k)$, the fluid level in the busy queue decreases at rate $\mu$ from $x_k$ to $x_k-\mu v_k$, after which it decreases at rate $\mu - \lambda$ from $x_k-\mu v_k$ to $0$.
For the idle queue, the fluid level increases at rate $\lambda$ from $0$ to $\lambda v_k$ over $[t_k,t_k + v_k)$, and then stays constant over $[t_k + v_k, t_{k+1})$. A little algebra shows that
$$ t_{k+1}-v_k-t_k = \frac{x_k-\mu v_k}{\mu - \lambda}.$$
Therefore, the cost during the $k$th  cycle is:
\begin{equation}
\label{eq:cost2q}
c\left(\frac{2x_k-\mu v_k}{2}v_k+\frac{(x_k-\mu v_k)^2}{2(\mu-\lambda)}\right)+d\left(\frac{\lambda v_k^2}{2} +\frac{\lambda v_k(x_k-\mu v_k)}{\mu -\lambda}\right) .
\end{equation}
Furthermore, we obtain the following dynamics:
\begin{align*}
x_{k+1}&=\lambda v_k, &
\intertext{subject to the constraint}
0&\leq \mu v_k \leq x_k.
\end{align*}
Note that using the dynamics we have
\begin{equation}
\label{eq:zeroq2}
\sum_{k=0}^\infty (x_k^2-\lambda^2v_k^2)=
\sum_{k=0}^\infty (x_k^2-x_{k+1}^2)=x_0^2 .
\end{equation}
So subtracting $\frac{d}{2(\mu -\lambda)}$ times the left hand side of \eqref{eq:zeroq2} (which is like subtracting a constant) from the sum of \eqref{eq:cost2q} over all $k$, we need to solve the following problem:
\begin{align*}
\min_{v_k}\quad & \frac{c-d}{2(\mu-\lambda)}\sum_{k=0}^\infty (x_k^2-2\lambda x_kv_k+\lambda \mu v_k^2)
\intertext{subject to}
x_{k+1}&=\lambda v_k ,\\
0&\leq \mu v_k\leq x_k .
\end{align*}
To eliminate the product $\lambda x_k v_k$ we define the new variable
$$
u_k=x_k- \mu v_k ,
$$
which allows us to rewrite our problem as \eqref{eq:lqrq2}.
\end{proof}

The next theorem presents the solution to the above problem.
\begin{theorem}
The optimal solution to the constrained LQR in Theorem~\ref{th:lqr} is given by
\begin{equation}
\label{eq:optcon}
u_k=\frac{-(1-\rho)+\sqrt{(1-\rho)(1+3\rho)}}{1+\rho+\sqrt{(1-\rho)(1+3\rho)}}x_k, \;\;\;k \ge 0
\end{equation}
and
\begin{equation} \label{eq:xkp1}
x_{k+1} = \beta x_k, \;\;\; k \ge 0,
\end{equation}
where
\begin{equation} \label{eq:bet}
\beta = \frac{2\rho}{1+\rho+\sqrt{(1-\rho)(1+3\rho)}}.
\end{equation}
\end{theorem}

\begin{proof}
Without constraint \eqref{eq:constraint}  problem
\eqref{eq:lqrq2} is the standard (infinite-horizon discrete-time) LQR problem. Initially we ignore  constraint \eqref{eq:constraint}. Then the
solution to the optimal control problem \eqref{eq:lqrq2} is given by
\begin{align*}
u_k&=-fx_k
\intertext{where}
f&=(r+bpb)^{-1}bpa
\intertext{and $p$ is the (unique) nonnegative solution of the discrete time algebraic Ricatti equation:}
p&=q+a(p-pb(r+bpb)^{-1}bp)a
\intertext{where}
a=\rho, & b=-\rho, \\
q=(1-\rho), & r=\rho.
\end{align*}
That is, the optimal solution is given by
\[u_k=\frac{-(1-\rho)+\sqrt{(1-\rho)^2+4\rho(1-\rho)}}{1+\rho+\sqrt{(1-\rho)(1+3\rho)}}x_k ,
\]
which can written as \ref{eq:optcon}.
Recall that we ignored constraint \eqref{eq:constraint}. However,
solution \eqref{eq:optcon} satisfies \eqref{eq:constraint}, so
\eqref{eq:optcon} is also the optimal solution for  problem \eqref{eq:lqrq2}
including  constraint \eqref{eq:constraint}. Using \ref{eq:optcon} in $x_{k+1}=\lambda v_k$ and $u_k=x_k- \mu v_k$ we get \eqref{eq:xkp1}.
\end{proof}

One can show that $0 \le \beta < 1$. Thus the amount of the fluid at switch-over times decreases geometrically to zero. The optimal policy goes through an infinite number of switch-overs before the system becomes empty. The next theorem specifies the optimal policy implied by the above theorem.

\begin{theorem} \label{th:optpol}
Let $\alpha$ be as in \eqref{eq:alp2}. It is optimal to route all incoming fluid to the idle queue at time $t$ if
\begin{equation} \label{eq:swc}
y(t) < \alpha x(t)
\end{equation}
and all incoming fluid to the busy queue otherwise.
\end{theorem}

\begin{proof}
Notice that at time $v_k$, the fluid level of the busy queue reduces to
$$
q_1=\frac{-(1-\rho)+\sqrt{(1-\rho)(1+3\rho)}}{1+\rho+\sqrt{(1-\rho)(1+3\rho)}}x_k
$$
and that of the idle queue increases from $0$ to
$$
q_2=\rho\left(x_k-\frac{-(1-\rho)+\sqrt{(1-\rho)(1+3\rho)}}{1+\rho+\sqrt{(1-\rho)(1+3\rho)}}x_k\right)=
\frac{2\rho}{1+\rho+\sqrt{(1-\rho)(1+3\rho)}}x_k.
$$
Thus \eqref{eq:swc} is satisfied for $t \in [t_k + v_k, t_{k+1})$, and it is not satisfied for $t \in [t_k, t_k + v_k)$. At $t=t_k + v_k$ it is satisfied at equality. This proves the theorem.
\end{proof}

Note that the switching curve in \eqref{eq:swc} matches the conjecture in \eqref{eq:con}!

\begin{remark}
{\rm
In the derivations of the previous three theorems, we have assumed that $y(0) = 0$, but this is not a restriction. To see this consider two cases:
\begin{enumerate}
\item[i.]
$0 < y(0) = y_0 \le \alpha x_0$: Consider a system starting at time $\tau = - y(0)/ \lambda$, in state $x(\tau) = x_0 + \mu \tau$ and $y(\tau) = 0$. Then following the optimal policy of Theorem~\ref{th:optpol} from time $t \ge \tau$ will bring the system to state $x(0) = x_0$ and $y(0) = y_0$. Hence the same optimal policy will continue to hold for $t \ge 0$ from the principle of optimality.
\item[ii.]
$0 < \alpha x_0 < y(0) = y_0$: In this case define
\[ \tau_1 = \frac{y(0)-x_0}{\alpha(\mu-\lambda)}, \quad \tau_2 = \frac{y(0)}{\lambda}.\]
Now consider a system starting at time $\tau = - (\tau_1 + \tau_2)$, in state $x(\tau) = x_0 +\alpha(\mu - \lambda)\tau_1 + \mu \tau_2$ and $y(\tau) = 0$. Then following the optimal policy of Theorem~\ref{th:optpol} from time $t \ge \tau$ will bring the system to state $x(0) = x_0$ and $y(0) = y_0$. Hence the same optimal policy will continue to hold for $t \ge 0$ from the principle of optimality.
\end{enumerate}
}
\end{remark}

\section{Numerical Example and Conclusions}
\label{sec:num}

In this section  we present a numerical example with the following parameters:
\[ \lambda = 0.3, \;\; \mu = 0.7, \;\; c=6, \;\; d=1.\]
The three switching curves for the three policies are shown in Figure~\ref{fig:optpol}. The bottom curve corresponds to the switching curve $h$ of the individually optimal policy. It is optimal to join the busy queue in all states $(i,j)$ that lie above this curve. We have numerically observed that as $d$ approaches zero, the switching curve moves up and it reaches $h(i) = i$ when $d=0$, that is, the individually optimal policy is to join the shortest queue when $d=0$. On the other hand, as $d$ increases, the switching curve moves down, and reaches $h(i)=0$ when $d \ge c$: that is, the optimal policy is to always join the busy queue.

\begin{figure}[ht]
\centering
\includegraphics[width=0.9\linewidth]{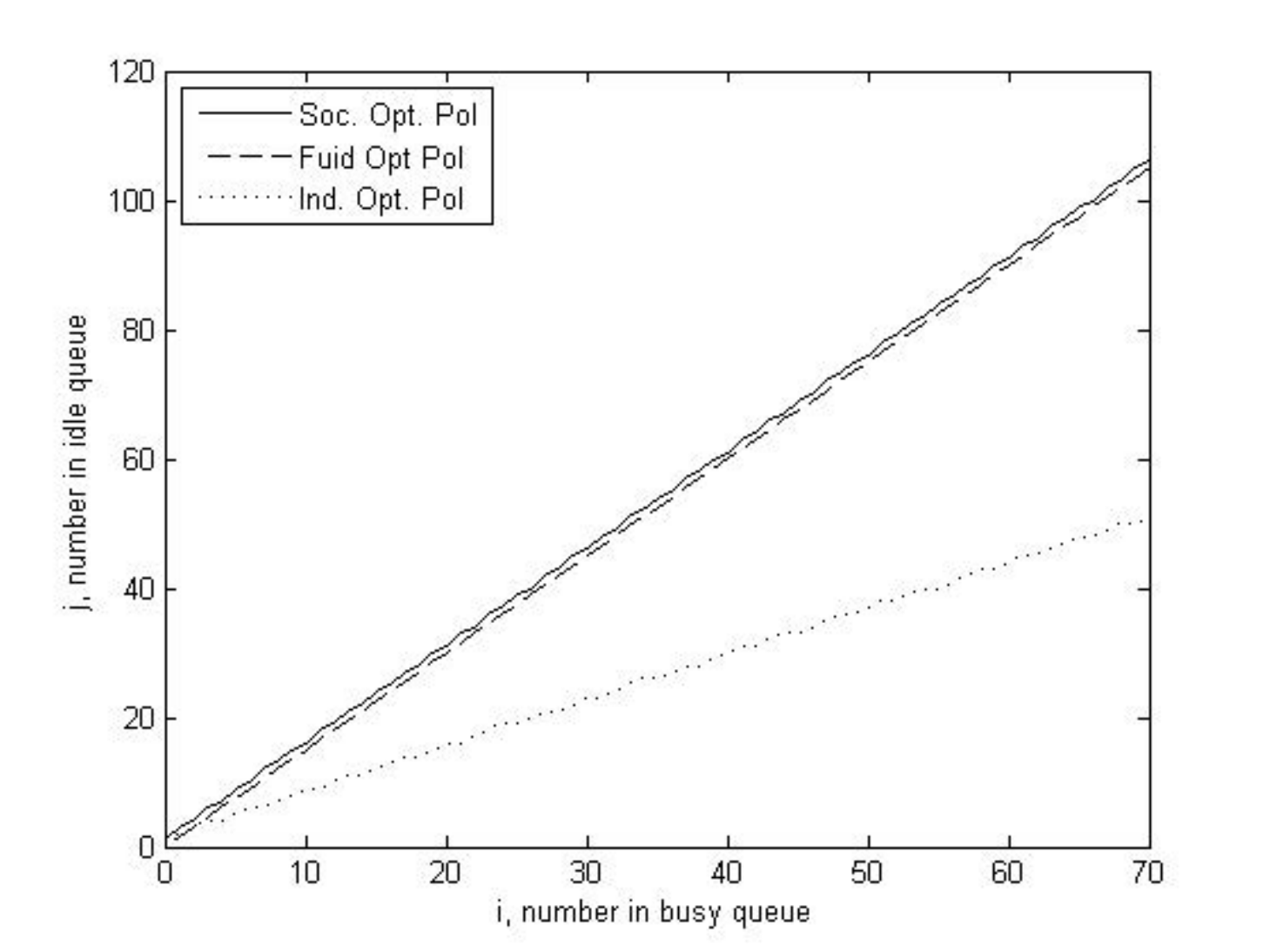}
\caption{The switching curves for the three policies.}\label{fig:optpol}
\end{figure}

The top curve in Figure~\ref{fig:optpol} corresponds to the switching curve $g$ of the socially optimal policy. As we discussed before, this is independent of $c$ and $d$, as long as $c > d$. The middle curve corresponds to the fluid switching curve $\alpha i$, where $\alpha$ is as in \eqref{eq:alp2}. The fluid curve is also independent of $c$ and $d$ as long as $c > d$. It is interesting to see that the two curves are quite close. For both policies, it is optimal to join the busy queue in all states $(i,j)$ that lie above the curve.

Observe that both the socially optimal and the fluid switching curves are above the  curve $j=i$, while the individually optimal curve is below it. This observation for the individually optimal policies was proved in Theorem~\ref{theorem:BusyForSure}. It follows for the fluid policy, because we know that $\alpha > 1$. We have not been able to prove it for the socially optimal policy.

Figure~\ref{fig:optpol} also shows that more customers join the busy queue under the individually optimal policy than under the socially optimal policy. This is consistent with the general observation in other queueing systems, and it is a result of externalities: in individually optimal policies, customers are selfish and ignore the cost their decision imposes on other customers, and hence tend to over-utilize the resources.

It would be interesting to formally prove these observations, but we leave that as future work. In the current paper we considered a simple exhaustive polling system with two queues, identical exponential service times, and no switch-over times and switch-over costs. Clearly, several extensions are possible: to more than two queues, non-identical exponential service times, general service times, service policies other than exhaustive service, non-zero switch-over times or costs, and so on. Each of these extensions makes the analysis harder, since the expressions for the expected queue lengths become more involved.

\end{document}